\documentclass{amsart}
\usepackage{setspace}
\usepackage{a4}
\usepackage{amsthm}
\usepackage{latexsym}
\usepackage{amsfonts}
\usepackage{graphicx}
\usepackage{textcomp}
\usepackage{cite}
\usepackage{enumerate}
\usepackage{amssymb}
\usepackage{hyperref}
\usepackage{amsmath}
\usepackage{tikz}
\usepackage[mathscr]{euscript}
\usepackage{mathtools}
\newtheorem{theorem}{Theorem}[section]

\title{This is the title}
\usepackage{fancyhdr}

\begin{document}
\hrule\hrule\hrule\hrule\hrule
\vspace{0.3cm}	
\begin{center}
{\bf\large{{Non-Archimedean Massera-Schaffer-Maligranda-Pecaric-Rajic Inequality}}}\\
\vspace{0.3cm}
\hrule\hrule\hrule\hrule\hrule
\vspace{0.3cm}
\textbf{K. Mahesh Krishna}\\
School of Mathematics and Natural Sciences\\
Chanakya University Global Campus\\
NH-648, Haraluru Village\\
Devanahalli Taluk, 	Bengaluru  North District\\
Karnataka State, 562 110, India\\
Email: kmaheshak@gmail.com\\

Date: \today
\end{center}

\hrule\hrule
\vspace{0.5cm}
\textbf{Abstract}: Massera and Schaffer [\textit{Ann. Math. (2), 1958}] derived a breakthrough upper bound for the Clarkson angle between two nonzero vectors in a normed linear space, which was later improved by Maligranda [\textit{Am. Math. Mon., 2006}]. Pecaric and Rajic [\textit{Math. Inequal. Appl., 2007}] extended Maligranda's inequality to finitely many nonzero vectors. We derive a non-Archimedean version of  Massera-Schaffer-Maligranda-Pecaric-Rajic inequality.

\hrule

\hrule
\section{Introduction}

Let $\mathcal{X}$ be a normed linear space (NLS). Recall that for $x, y \in \mathcal{X}\setminus\{0\}$, the \textbf{Clarkson angle} \cite{CLARKSON} between $x$ and $y$ is defined as 
\begin{align*}
	\alpha[x,y]\coloneqq \left\|\frac{x}{\|x\|}-\frac{y}{\|y\|}\right\|.
\end{align*}
Important Massera-Schaffer inequality \cite{MASSERASCHAFFER} says that 
\begin{align}\label{MS}
	\alpha[x,y]=\left\|\frac{x}{\|x\|}-\frac{y}{\|y\|}\right\|\leq \frac{2\|x-y\|}{\max\{\|x\|, \|y\|\}}, \quad \forall x, y \in \mathcal{X}\setminus\{0\}.
\end{align}
Maligranda improved Inequality (\ref{MS}) in 2006 \cite{MALIGRANDA}. 
\begin{theorem} \cite{MALIGRANDA, MERCER} \label{DWMT} (\textbf{Massera-Schaffer-Maligranda Inequality})
Let $\mathcal{X}$ be a NLS. Then for all $x, y \in \mathcal{X}\setminus\{0\}$, 
\begin{align}\label{MI}
	\alpha[x,y]=\left\|\frac{x}{\|x\|}-\frac{y}{\|y\|}\right\|\leq \frac{\|x-y\|+|\|x\|-\|y\||}{\max\{\|x\|, \|y\|\}}\leq \frac{2\|x-y\|}{\max\{\|x\|, \|y\|\}}.	
\end{align}
\end{theorem}
We note that,  in 1964 Dunkl and Williams \cite{DUNKLWILLIAMS} independently showed that 
\begin{align}\label{DW}
\left\|\frac{x}{\|x\|}-\frac{y}{\|y\|}\right\|\leq \frac{4\|x-y\|}{\|x\|+\|y\|}, \quad \forall x, y \in \mathcal{X}\setminus\{0\}.	
\end{align}
Note that 
\begin{align*}
	\frac{2\|x-y\|}{\max\{\|x\|, \|y\|\}}\leq \frac{4\|x-y\|}{\|x\|+\|y\|}, \quad \forall x, y \in \mathcal{X}\setminus\{0\}.	
\end{align*}
Therefore Inequality (\ref{MS}) improves Inequality (\ref{DW}). 
Inequality (\ref{DW}) is famously known as Dunkl-Williams inequality in the literature. In 2007, Pecaric and Rajic extended Inequality (\ref{MI}) to finitely many nonzero elements \cite{PECARICRAJIC}. 
\begin{theorem} \cite{PECARICRAJIC} (\textbf{Massera-Schaffer-Maligranda-Pecaric-Rajic Inequality}) \label{DWMPRT}
Let $\mathcal{X}$ be a NLS and $n \in \mathbb{N}$. Then for all $x_1, \dots, x_n \in \mathcal{X}\setminus\{0\}$, 
\begin{align*}
\left\|\sum_{j=1}^{n}\frac{x_j}{\|x_j\|}\right\|\leq \min_{1\leq k \leq n}\left\{\frac{1}{\|x_k\|}\left(\left\|\sum_{j=1}^{n}x_j\right\|+\sum_{j=1}^{n}\big|\|x_j\|-\|x_k\|\big|\right)\right\}.
\end{align*}
\end{theorem}
It  is important to ask what are non-Archimedean versions of Theorems \ref{DWMT} and \ref{DWMPRT}? We answer the question by deriving non-Archimedean Massera-Schaffer-Maligranda-Pecaric-Rajic Inequality (Theorem \ref{NAMSMPRT}).

\section{Non-Archimedean Massera-Schaffer-Maligranda-Pecaric-Rajic Inequality}

Let $\mathbb{K}$ be a field. A map $|\cdot|: \mathbb{K} \to [0, \infty)$ is said to be a non-Archimedean  valuation  if following conditions holds.
\begin{enumerate}[\upshape(i)]
	\item If $\lambda  \in \mathbb{K}$ is such that $|\lambda|=0$, then $\lambda=0$.
	\item $|\lambda \mu|=|\lambda||\mu|$ for all $\lambda, \mu  \in \mathbb{K}$.
	\item (Ultra-triangle inequality) $|\lambda+\mu|\leq \max\{|\lambda|, |\mu|\}$ for all $\lambda, \mu \in \mathbb{K}$.
\end{enumerate}
In this case, $\mathbb{K}$ is called as non-Archimedean valued field \cite{SCHIKHOF}.  Let $\mathcal{X}$ be a vector space over a non-Archimedean valued field $\mathbb{K}$ with valuation $|\cdot|$. A map $\|\cdot\|: \mathcal{X} \to [0, \infty)$ is said to be a non-Archimedean norm if following conditions holds.
\begin{enumerate}[\upshape(i)]
	\item If $x \in \mathcal{X}$ is such that $\|x\|=0$, then $x=0$.
	\item $\|\lambda x\|=|\lambda|\|x\|$ for all $\lambda \in \mathbb{K}$, for all $x \in \mathcal{X}$.
	\item (Ultra-norm inequality) $\|x+y\|\leq \max\{\|x\|, \|y\|\}$ for all $x,y \in \mathcal{X}$.
\end{enumerate}
In this case, $\mathcal{X}$ is called as non-Archimedean linear space (NALS) \cite{GARCIASCHIKHOF}.   Let $\mathcal{X}$ be a NALS over $\mathbb{K}$. Let  $x, y \in \mathcal{X}\setminus\{0\}$ with $\|x\|, \|y\| \in \mathbb{K}$. We define the \textbf{non-Archimedean Clarkson angle}  between $x$ and $y$ as 
\begin{align*}
	\alpha[x,y]\coloneqq \left\|\frac{x}{\|x\|}-\frac{y}{\|y\|}\right\|.
\end{align*}
Non-Archimedean version of Theorem \ref{DWMT} now reads as follows. 
\begin{theorem} (\textbf{Non-Archimedean Massera-Schaffer Inequality})
Let $\mathcal{X}$ be a NALS over $\mathbb{K}$. Then for all $x, y \in \mathcal{X}\setminus\{0\}$ with $\|x\|, \|y\| \in \mathbb{K}$ it holds
\begin{align*}
	\alpha[x,y]= \left\|\frac{x}{\|x\|}-\frac{y}{\|y\|}\right\|\leq \min \left\{\frac{1}{\big|\|x\|\big|}\max\left\{\|x-y\|, \left|1-\frac{\|x\|}{\|y\|}\right|\|y\|\right\}, \frac{1}{\big|\|y\|\big|}\max\left\{\|x-y\|, \left|1-\frac{\|y\|}{\|x\|}\right|\|x\|\right\}\right\}.
\end{align*}
\end{theorem}
\begin{proof}
	Let $x, y \in \mathcal{X}\setminus\{0\}$ with $\|x\|, \|y\| \in \mathbb{K}$. Then 
		\begin{align*}
	 \left\|\frac{x}{\|x\|}-\frac{y}{\|y\|}\right\|&= \left\|\frac{x-y}{\|x\|}+\left(\frac{1}{\|x\|}-\frac{1}{\|y\|}\right)y\right\|\\
	 &\leq \max \left\{\left\|\frac{x-y}{\|x\|}\right\|, \left\|\left(\frac{1}{\|x\|}-\frac{1}{\|y\|}\right)y\right\|\right\}\\
	  &\leq \max \left\{\frac{\|x-y\|}{\big|\|x\|\big|}, \left|\frac{1}{\|x\|}-\frac{1}{\|y\|}\right|\|y\|\right\}\\
	  &= \max \left\{\frac{\|x-y\|}{\big|\|x\|\big|}, \frac{1}{\left|\|x\|\right|}\left|1-\frac{\|x\|}{\|y\|}\right|\|y\|\right\}\\
	\end{align*}
and 
\begin{align*}
 \left\|\frac{x}{\|x\|}-\frac{y}{\|y\|}\right\|&= \left\|\left(\frac{1}{\|x\|}-\frac{1}{\|y\|}\right)x+\frac{x-y}{\|y\|}\right\|\\
&\leq \max \left\{\left\|\left(\frac{1}{\|x\|}-\frac{1}{\|y\|}\right)x\right\|, \left\|\frac{x-y}{\|y\|}\right\|\right\}\\
&\leq \max \left\{\left|\frac{1}{\|x\|}-\frac{1}{\|y\|}\right|\|x\|, \frac{\|x-y\|}{\big|\|y\|\big|}\right\}\\
&\leq \max \left\{\frac{1}{\big|\|y\|\big|}\left|1-\frac{\|y\|}{\|x\|}\right|\|x\|, \frac{\|x-y\|}{\big|\|y\|\big|}\right\}.
\end{align*}
Therefore 
\begin{align}\label{1}
 \left\|\frac{x}{\|x\|}-\frac{y}{\|y\|}\right\|\leq \frac{1}{\big|\|x\|\big|}\max\left\{\|x-y\|, \left|1-\frac{\|x\|}{\|y\|}\right|\|y\|\right\}
\end{align}
and 
\begin{align}\label{2}
	\left\|\frac{x}{\|x\|}-\frac{y}{\|y\|}\right\|\leq \frac{1}{\big|\|y\|\big|}\max\left\{\|x-y\|, \left|1-\frac{\|y\|}{\|x\|}\right|\|x\|\right\}.
\end{align}
Inequalities (\ref{1}) and  (\ref{2}) give 
\begin{align*}
	\left\|\frac{x}{\|x\|}-\frac{y}{\|y\|}\right\|&\leq\min \left\{\frac{1}{\big|\|x\|\big|}\max\left\{\|x-y\|, \left|1-\frac{\|x\|}{\|y\|}\right|\|y\|\right\}, \frac{1}{\big|\|y\|\big|}\max\left\{\|x-y\|, \left|1-\frac{\|y\|}{\|x\|}\right|\|x\|\right\}\right\}.
\end{align*}
\end{proof}
Note the additional assumption $\|x\|, \|y\| \in \mathbb{K}$ in the previous theorem. This is necessary because, since the norm is a real number, we generally cannot guarantee that it belongs to the given non-Archimedean field. We now derive the non-Archimedean version of Theorem \ref{DWMPRT}.
\begin{theorem} \label{NAMSMPRT} (\textbf{Non-Archimedean Massera-Schaffer-Maligranda-Pecaric-Rajic Inequality})
	Let $\mathcal{X}$ be a NALS over $\mathbb{K}$ and $n \in \mathbb{N}$. Then for all $x_1, \dots, x_n \in \mathcal{X}\setminus\{0\}$ with $\|x_1\|, \dots, \|x_n\| \in \mathbb{K}$ it holds
	\begin{align*}
		\left\|\sum_{j=1}^{n}\frac{x_j}{\|x_j\|}\right\|\leq \min_{1\leq k \leq n}\left\{\frac{1}{\big|\|x_k\|\big|} \max \left\{\left\|\sum_{j=1}^{n}x_j\right\|, \max_{1\leq j \leq n} \left|1-\frac{\|x_k\|}{\|x_j\|}\right|\|x_j\|\right\}\right\}.
	\end{align*}
\end{theorem}
\begin{proof}
Let $x_1, \dots, x_n \in \mathcal{X}\setminus\{0\}$ with $\|x_1\|, \dots, \|x_n\| \in \mathbb{K}$. Let $1\leq k \leq n$ be fixed. Then 
\begin{align*}
	\left\|\sum_{j=1}^{n}\frac{x_j}{\|x_j\|}\right\|&=\left\|\frac{x_k}{\|x_k\|}+\sum_{j=1, j \neq k}^{n}\frac{x_j}{\|x_j\|}\right\|\\
	&=\left\|\sum_{j=1}^{n}\frac{x_j}{\|x_k\|}-\sum_{j=1, j \neq k}^{n}\frac{x_j}{\|x_k\|}+\sum_{j=1, j \neq k}^{n}\frac{x_j}{\|x_j\|}\right\|\\
	&=\left\|\sum_{j=1}^{n}\frac{x_j}{\|x_k\|}-\sum_{j=1, j \neq k}^{n}\left(\frac{1}{\|x_k\|}-\frac{1}{\|x_j\|}\right)x_j\right\|\\
	&=\left\|\sum_{j=1}^{n}\frac{x_j}{\|x_k\|}-\sum_{j=1}^{n}\left(\frac{1}{\|x_k\|}-\frac{1}{\|x_j\|}\right)x_j\right\|\\
	&\leq \max \left\{\left\|\sum_{j=1}^{n}\frac{x_j}{\|x_k\|}\right\|, \left\|\sum_{j=1}^{n}\left(\frac{1}{\|x_k\|}-\frac{1}{\|x_j\|}\right)x_j\right\|\right\}\\
	&=\max \left\{\frac{1}{\big|\|x_k\|\big|}\left\|\sum_{j=1}^{n}x_j\right\|, \left\|\sum_{j=1}^{n}\left(\frac{1}{\|x_k\|}-\frac{1}{\|x_j\|}\right)x_j\right\|\right\}\\
	&\leq \max \left\{\frac{1}{\big|\|x_k\|\big|}\left\|\sum_{j=1}^{n}x_j\right\|, \max_{1\leq j \leq n} \left\|\left(\frac{1}{\|x_k\|}-\frac{1}{\|x_j\|}\right)x_j\right\|\right\}\\
	&= \max \left\{\frac{1}{\big|\|x_k\|\big|}\left\|\sum_{j=1}^{n}x_j\right\|, \frac{1}{\big|\|x_k\|\big|}\max_{1\leq j \leq n} \left|1-\frac{\|x_k\|}{\|x_j\|}\right|\|x_j\|\right\}\\
	&=\frac{1}{\big|\|x_k\|\big|} \max \left\{\left\|\sum_{j=1}^{n}x_j\right\|, \max_{1\leq j \leq n} \left|1-\frac{\|x_k\|}{\|x_j\|}\right|\|x_j\|\right\}.
\end{align*}

By varying $k$ and taking minimum in the right side of previous inequality gives 
\begin{align*}
		\left\|\sum_{j=1}^{n}\frac{x_j}{\|x_j\|}\right\|\leq \min_{1\leq k \leq n}\left\{\frac{1}{\big|\|x_k\|\big|} \max \left\{\left\|\sum_{j=1}^{n}x_j\right\|, \max_{1\leq j \leq n} \left|1-\frac{\|x_k\|}{\|x_j\|}\right|\|x_j\|\right\}\right\}.
\end{align*}
\end{proof}
\section{Conclusions}
\begin{enumerate}
	\item In 1958, Massera and Schaffer derived a surprising upper bound for the Clarkson angle between two nonzero elements in a normed linear space \cite{MASSERASCHAFFER}.
	\item In 2006, Maligranda improved Massera-Schaffer inequality \cite{MALIGRANDA}.
	\item In 2007, Pecaric and Rajic extended Maligranda inequality for finitely many nonzero elements \cite{PECARICRAJIC}. 
	\item In this article, we introduced the notion of non-Archimedean Clarkson angle and derived   non-Archimedean version of Massera-Schaffer-Maligranda-Pecaric-Rajic inequality. 
\end{enumerate}

 \bibliographystyle{plain}
 \bibliography{reference.bib}

\begin{thebibliography}{1}

\bibitem{CLARKSON}
James~A. Clarkson.
\newblock Uniformly convex spaces.
\newblock {\em Trans. Am. Math. Soc.}, 40:396--414, 1936.

\bibitem{DUNKLWILLIAMS}
Charles~F. Dunkl and K.~S. Williams.
\newblock A simple norm inequality.
\newblock {\em Am. Math. Mon.}, 71:53--54, 1964.

\bibitem{MALIGRANDA}
Lech Maligranda.
\newblock Simple norm inequalities.
\newblock {\em Am. Math. Mon.}, 113(3):256--260, 2006.

\bibitem{MASSERASCHAFFER}
J.~L. Massera and J.~J. Sch{\"a}ffer.
\newblock Linear differential equations and functional analysis. {I}.
\newblock {\em Ann. Math. (2)}, 67:517--573, 1958.

\bibitem{MERCER}
Peter~R. Mercer.
\newblock A refined {C}auchy-{S}chwarz inequality.
\newblock {\em International Journal of Mathematical Education in Science and
  Technology}, 38(6):839--843, 2007.

\bibitem{PECARICRAJIC}
Josip Pe{\v{c}}ari{\'c} and Rajna Raji{\'c}.
\newblock The {Dunkl}-{Williams} inequality with {{\(n\)}} elements in normed
  linear spaces.
\newblock {\em Math. Inequal. Appl.}, 10(2):461--470, 2007.

\bibitem{GARCIASCHIKHOF}
C.~Perez-Garcia and W.~H. Schikhof.
\newblock {\em Locally convex spaces over non-{Archimedean} valued fields},
  volume 119 of {\em Camb. Stud. Adv. Math.}
\newblock Cambridge: Cambridge University Press, 2010.

\bibitem{SCHIKHOF}
W.~H. Schikhof.
\newblock {\em Ultrametric calculus. {An} introduction to $p$-adic analysis},
  volume~4 of {\em Camb. Stud. Adv. Math.}
\newblock Cambridge: Cambridge University Press, 2006.

\end{thebibliography}

\end{document}